\documentclass[12pt]{amsart}
\usepackage{amssymb, color}
\usepackage{amsmath}
\usepackage{amsthm}
 \font\tenmsb=msbm10 at 12pt \font\sevenmsb=msbm7 at 8pt \font\fivemsb=msbm5 at
6pt
\newfam\msbfam
\textfont\msbfam=\tenmsb \scriptfont\msbfam=\sevenmsb \scriptscriptfont\msbfam=\fivemsb

\newcommand{\beq}{\begin{equation}}

\newcommand{\eeq}{\end{equation}}

\def \R{\mathbb{R}}

\def\qed{%
\mbox{ }%
\nolinebreak%
\hfill%
\rule{2mm} {2mm}%
\medbreak%
\par%
}
\newtheorem{thm}{Theorem}[section]
\newtheorem{lem}[thm]{Lemma}

\theoremstyle{definition}

\begin{document}
\title[Delaunay entire solutions]
{On Delaunay solutions of a biharmonic elliptic equation with critical exponent}
\author{Zongming Guo}
\address{Department of Mathematics, Henan Normal University, Xinxiang, 453007, China}
\email{gzm@htu.cn}
\author{Xia Huang}
\address{Center for Partial Differential Equations, East China Normal
University, Shanghai, 200241, China}
\email{xhuang1209@gmail.com}
\author{Liping Wang}
\address{Department of Mathematics, Shanghai Key Laboratory of Pure Mathematics and Mathematical Practice, East China Normal
University, Shanghai, 200241, China}
\email{lpwang@math.ecnu.edu.cn}
\author{Juncheng Wei}
\address{Department of Mathematics, University of British Columbia, Vancouver, B.C. Canada V6T 1Z2}
\email{jcwei@math.ubc.ca}

\begin{abstract}
We are interested in the qualitative properties of positive entire solutions $u \in C^4 (\R^n \backslash \{0\})$ of the equation
\begin{equation}
\label{0.0}
\Delta^2 u=u^{\frac{n+4}{n-4}} \;\;\mbox{in $\R^n \backslash \{0\}$ and 0 is a non-removable singularity of $u(x)$}.
\end{equation}
It is known from [Theorem 4.2, \cite{L}] that any positive entire solution $u$ of \eqref{0.0} is radially symmetric with respect to $x=0$, i.e. $u(x)=u(|x|)$, and equation \eqref{0.0} also  admits a
special positive entire solution $u_s (x)=\Big(\frac{n^2 (n-4)^2}{16} \Big)^{\frac{n-4}{8}} |x|^{-\frac{n-4}{2}}$. We first show  that $u-u_s$ changes signs infinitely many times in
$(0, \infty)$ for any positive singular entire solution $u \not \equiv u_s$ in $\R^N \backslash \{0\}$ of \eqref{0.0}. Moreover, equation \eqref{0.0} admits a positive entire singular solution $u(x) \; (=u(|x|)$ such that the scalar curvature of the conformal metric with conformal factor $u^{\frac{4}{n-4}}$ is positive and $v(t):=e^{\frac{n-4}{2} t} u(e^t)$ is $2T$-periodic
with suitably large $T$. It is still open that
$v(t):=e^{\frac{n-4}{2} t} u(e^t)$ is  periodic for any   positive entire solution $u(x)$ of \eqref{0.0}.
\end{abstract}

\subjclass{Primary 35B45, 35R35; Secondary 35J30, 35J40}

\keywords{Biharmonic equations, critical exponent, positive singular solutions,
positive periodic solutions, scalar curvature}

\maketitle

\section{Introduction}
\setcounter{equation}{0}

We are interested in the qualitative properties of positive singular solutions of the equation
\begin{equation}
\label{0.1}
\Delta^2 u=u^{p} \qquad  \mbox{in} \quad \R^n
\end{equation}
where $p=\frac{n+4}{n-4}, n \geq 5$.

Equation (\ref{0.1}) arises in both physics and geometry. There are many results about the classification of solutions to (\ref{0.1}). If $1<p<\frac{n+4}{n-4}$, all nonnegative regular solutions to (\ref{0.1}) are trivial. While for  $p=\frac{n+4}{n-4}$, any positive regular solution has the form
\begin{equation}
\label{0.2}
u_\lambda (x)=c_n \Big(\frac{\lambda}{1+\lambda^2|x-x_0|^2} \Big)^{\frac{n-4}{2}}, \qquad x_0\in\R^n
\end{equation}
where $c_n=[n(n-4)(n-2)(n+2)]^{-\frac{n-4}{8}}$ and $\lambda \in (0, \infty)$, see  \cite{L}  and \cite{WX}. Obviously $u_\lambda \in C^4 (\R^n)$ for any $\lambda>0$. Under the transformations:
$$v_\lambda (t):=|x-x_0|^{\frac{n-4}{2}} u_\lambda (|x-x_0|), \;\;\; t=\log (\lambda |x-x_0|),$$
a simple calculation implies that
\begin{equation}
\label{0.3}
v_\lambda (t)=c_n \Big(2 \cosh (t) \Big)^{-\frac{n-4}{2}} \;\; \mbox{for $t \in (-\infty, \infty)$}.
\end{equation}
It can be seen from \eqref{0.3} that, for each $\lambda>0$, $v_\lambda (t)$ satisfies
$$v_\lambda (-\infty)=v_\lambda (\infty)=0$$
and there is only one maximum point $t_0 =0$ of $v_\lambda (t)$.

In this paper, we are interested in qualitative properties of positive entire solutions $u$ of \eqref{0.1} with a non-removable singularity at $x=0$, i.e., $u$ satisfies the problem
\begin{equation}
\label{1.1}
\left \{ \begin{array}{l} \Delta^2 u=u^{\frac{n+4}{n-4}} \;\;\; \mbox{in $\R^n \backslash \{0\}$},\\
0 \;\; \mbox{is a non-removable singularity of $u(x)$}.
\end{array} \right.
\end{equation}

It is known from \cite{L} that if $u \in C^4 (\R^n \backslash \{0\})$ is a positive entire solution to \eqref{1.1} and 0 is a non-removable singularity of $u$, then $u$ is radially symmetric with respect to $x=0$, i.e. $u(x)=u(|x|)$. We recall the following  theorem.
\begin{thm}(Theorem 4.2 of \cite{L}) Suppose $u$ is a positive smooth solution of
$$
\Delta^2 u=u^p \;\;\; \mbox{in $\R^n \backslash \{0\}$},
$$
where $1<p\leq\frac{n+4}{n-4}$. Assume $0$ is a non-removable singularity of $u$, then $u$ is radially symmetric with respect to the origin.
\end{thm}

In the second order case, it is known from Caffarelli, Gidas and Spruck \cite{CGS} that if $u \in C^2 (\R^n \backslash \{0\})$ is a positive solution of the problem:
\begin{equation}
\label{1.2}
\left \{ \begin{array}{l}
-\Delta u=u^{\frac{n+2}{n-2}} \qquad \mbox{in}\ \ \R^n \backslash \{0\}, \quad  n \geq 3,\\
 0 \;\; \mbox{is a non-removable singularity of $u(x)$},
\end{array} \right.
\end{equation}
then $u$ is radially symmetric with respect to the origin. Moreover, under the transformations:
$$ v(t)=|x|^{\frac{n-2}{2}} u(x), \;\;\; t=\log |x|,$$
 $v(t)$ is a periodic function of $t$ in $\mathbb{R}$ (see also \cite{HLW}). These periodic solutions are called Delaunay type solutions in geometry. It is known that  Delaunay type solutions play vital role in the construction and classification  of singular Yamabe problem (\cite{MP1, MP2, MP3, MP4}). 

\medskip

An interesting question is whether the qualitative properties of the positive entire solutions of
the second order problem are still true  for the fourth order problem. More precisely,  if $u \in C^4 (\R^n \backslash \{0\})$ is a positive entire solution of \eqref{1.1}, under the transformations:
\begin{equation}
\label{1.3}
v(t)=|x|^{\frac{n-4}{2}} u(|x|), \;\;\; t=\log |x|,
\end{equation}
then $v(t)$ satisfies the equation
\begin{equation}
\label{1.4}
v^{(4)}(t)+K_2 v''(t)+K_0 v(t)=v^{\frac{n+4}{n-4}}(t),
\end{equation}
where
\begin{equation}
\label{1.5}
K_2=-\frac{n^2-4n+8}{2}, \;\; K_0=\frac{n^2(n-4)^2}{16}.
\end{equation}

\medskip

A natural question is: {\it Is $v(t)$ a periodic function of $t \in \mathbb{R}$?}

\medskip

In this paper we will see that there is a basic difference between the second order and the fourth order cases. It is easily seen that $v_s \equiv K^{\frac{n-4}{8}}_0$ in $\mathbb{R}$ is a constant solution to \eqref{1.4}, since
$$
u_s(r)=K^{\frac{n-4}{8}}_0 r^{-\frac{n-4}{2}}
$$
is a positive entire solution to \eqref{1.1}. We will prove that $v(t)$ in \eqref{1.3} oscillates infinitely many times in $t \in \mathbb{R}$ around the constant solution $v_s$ provided that $v \not \equiv v_s$ in $\mathbb{R}$, which means that $u(r)-u_s (r)$ changes signs infinitely many times in $(0,\infty)$ provided $u \not \equiv u_s$ in $(0, \infty)$. Furthermore, the existence of positive $2T$-periodic solution $v(t)$ of \eqref{1.4} will be given provided $T>0$ suitably large. It is still open that any solution  $v(t)$ of \eqref{1.4} is a periodic function.

Our main results of this paper are the following theorems.

\begin{thm}
\label{t1.1}
Assume $n \geq 5$, any positive entire solution $u\in C^4(\R^n\backslash \{0\})$ to problem \eqref{1.1}  is radially symmetric with respect to the origin. Moreover, $u(r)-u_s (r)$ changes signs infinitely many times in $(0,\infty)$ provided $u \not \equiv u_s$ in $(0, \infty)$.
\end{thm}

Theorem \ref{t1.1} shows that the equation $v'(t)=0$ admits infinitely many roots in $\mathbb{R}$
provided $v(t):=e^{\frac{n-4}{2} t} u(e^t)$. If there are two roots $-\infty<t_0<t_1<\infty$ such that $v'''(t_0)=v'''(t_1)=0$, we can easily see from \eqref{1.4} that $v(t)$ is a periodic solution to \eqref{1.4}. However, for arbitrary solution $v(t)$ of \eqref{1.4},  it is difficult to find  $t_*$ satisfying  $v'(t_*)=0$ and  $v'''(t_*)=0$. We have succeeded in seeking the periodic solution to equation \eqref{1.4} via variational methods.

\begin{thm}
\label{t1.2}
For any suitably large $T>0$, there exists a positive $2T$-periodic solution to equation \eqref{1.4}.
\end{thm}

As far as we know, except the well-known radial singular solution $K_1^{\frac1{p-1}}|x|^{-\frac4{p-1}}$ to equation (\ref{0.1}) with
$$
K_1=\frac{8\left[(n-2)(n-4)(p-1)^3+2(n^2-10n+20)(p-1)^2-16(n-4)(p-1)+32\right]}{(p-1)^4},
$$
Theorem \ref{t1.2} states that \eqref{1.1} admits solutions $u (|x|)$ with periodic $v(t):=e^{\frac{n-4}{2}t}u(e^t)$ for $t \in \R$.

Furthermore, for a solution $u(r)=v(t) r^{-\frac{n-4}{2}}$ of \eqref{1.1} where $t=\ln r$ and $v(t)$ is given in Theorem \ref{t1.2}, we can show that the scalar curvature
of the conformal metric with conformal factor $u^{\frac{4}{n-4}}$ is positive.  This result is an immediate consequence of the following pointwise inequality \eqref{1.9}. In fact, by the conformal change $g:=u^{\frac{4}{n-4}}g_0$ where $g_0$ is the usual Euclidean metric, the new scalar curvature becomes
\begin{equation}
R_g=-\frac{4(n-1)}{n-2} u^{-\frac{n+2}{n-4}} \Delta(u^{\frac{n-2}{n-4}}).
\end{equation}
\begin{thm}
\label{t1.3}
 If $v(t)=r^{\frac{n-4}{2}}u(e^t)$ is given by Theorem \ref{t1.2}, then the following pointwise inequality holds
\begin{equation}
\label{1.9}
-\Delta u\geq \sqrt{\frac{n(n-4)}{n^2-4}} u^{\frac{n}{n-4}}+\frac{2}{n-4}\frac{|\nabla u|^2}{u} \qquad \text{in} \ \ \R^n \backslash \{0\}.
\end{equation}
As a consequence, the scalar curvature $ R_g $ of the corresponding $u$ is positive.
\end{thm}

Such  pointwise inequalities have  important applications. For example Modica \cite{MO}  derived a similar Modica's estimate for Allen-Cahn equation. This point estimate has been used to study the De Giorgi's conjecture (1978) and analyze various semilinear equations and problems. Recently, Fazly, Wei and Xu \cite{FWX} obtained the inequality
\begin{equation}\label{bs}
-\Delta u \ge \sqrt{\frac{2}{p+1-\frac8{n(n-4)}}}|x|^{\frac{a}2}u^{\frac{p+1}2}+\frac2{n-4}\frac{|\nabla u|^2}{u}\qquad \mbox{in} \quad \R^n
\end{equation}
for all bounded positive solutions of $\Delta^2 u=|x|^au^p, \ p>1, a\ge 0$. For more details one may refer to \cite{FWX} and references therein. We should point out that  (\ref{bs}) holds for bounded solutions while in our case, the solutions have non-removable singularity at the point $0$, which may make $\int_{B_1(0)}|\Delta u|^2 dx$ to be unbounded. On the other hand, Theorem \ref{t1.2} and Theorem \ref{t1.3} imply that there is a relationship between the periodic solution and the corresponding scalar curvature. Based on Theorem \ref{t1.3}, it is natural to raise the following

\medskip

\noindent
{\bf Conjecture:} {\it If the scalar curvature corresponding to the solution $u(r)$ of (\ref{1.1}) is positive, then $v(t)=e^{\frac{n-4}2t}u(e^t)$ must be periodic.}

\medskip

Structure of positive radial entire solutions of the equation
\begin{equation}
\label{super}
\Delta^2 u=u^p \qquad  \mbox{in} \quad \R^n, \qquad  p>\frac{n+4}{n-4}, \quad n \ge 5
\end{equation}
has also been studied by many authors in these years. Gazzola and Grunau \cite{GG} proved the existence of positive regular radial solutions to (\ref{super}), denoted by $u_{a}$ satisfying
$$
u(0)=a, \qquad  u''(0)=b(a),\qquad u'(0)=u'''(0)=0
$$
where $b(a)$ is uniquely determined by $a$. Guo and Wei \cite{GW} got the qualitative properties of $u_a - K_1^{\frac1{p-1}}|x|^{-\frac4{p-1}}$. D\'avila, Dupaigne, Wang and Wei \cite{DDWW} proved that all stable or finite Morse index solutions to (\ref{super}) are trivial provided $p<p_c(n)$ where $p_c(n)=+\infty$ for $5 \leq n\leq 12$ and
 $p_c(n)=\frac{n+2-\sqrt{n^2+4-n\sqrt{n^2-8n+32}}}{n-6-\sqrt{n^2+4-n\sqrt{n^2-8n+32}}}$ for $n\ge 13$. Recently,  Guo, Wei and Yang \cite{GWY} obtained infinitely many nonradial singular solutions of (\ref{super}) by gluing method provided $\frac{n+3}{n-5}< p< p_c(n-1), \ n \ge 6.$  For more results and details one may refer to   the references therein.

In section 2, we present some qualitative properties of the entire solutions of \eqref{1.1} and the proof of Theorem \ref{t1.1}.  In section 3, we give the proof of Theorem \ref{t1.2} and the proof of Theorem \ref{t1.3} will be given in section 4 .

\section{Qualitative properties of positive entire solutions of \eqref{1.1}: Proof of Theorem \ref{t1.1}}
\setcounter{equation}{0}

In this section, we will present the proof of Theorem \ref{t1.1}.

Since any positive entire solution $u$ of \eqref{1.1} is radially symmetric, we write the equation of $u$ in radial coordinates $r=|x|$:
$$
u^{(4)}+\frac{2(n-1)}{r}u^{(3)}+\frac{(n-1)(n-3)}{r^2}u''-\frac{(n-1)(n-3)}{r^3} u'=u^{\frac{n+4}{n-4}}, \;\;\;\; \forall r>0.
$$
Set
\begin{equation}
\label{2.1}
v(t)=r^{\frac{n-4}{2}} u(r), \;\;\; t=\log r.
\end{equation}
Then $v \in C^4 (\mathbb{R})$ satisfies the ODE:
\begin{equation}
\label{2.2}
v^{(4)}(t)+K_2 v''(t)+K_0 v(t)=v(t)^{\frac{n+4}{n-4}}, \;\;\;\; t \in \mathbb{R},
\end{equation}
where $K_2$ and $K_0$ are given in \eqref{1.5}. Note that equation \eqref{2.2} admits two constant solutions: $v_0 \equiv 0$ and $v_s \equiv K^{\frac{n-4}{8}}_0$, which correspond to the radial solutions of equation \eqref{0.1}:
$$
u_0(r) \equiv 0,\;\;\; \quad u_s(r)=K^{\frac{n-4}{8}}_0r^{-\frac{n-4}{2}}.
$$

We now write \eqref{2.2} as a system in $\mathbb{R}^4$. By \eqref{2.1} we have
$$
u'(r)=r^{-\frac{n-2}{2}}\left[v'(t)-\frac{n-4}{2}v(t)\right],
$$
so that
$$
u'(r)=0 \;\; \Leftrightarrow \;\; v'(t)=\frac{n-4}{2}v(t).
$$
This fact suggests the definition
\begin{eqnarray*}
\left \{ \begin{array}{l}
w_1(t)=v(t),\\
w_2(t)=v'(t)-\frac{n-4}{2}v(t),\\
w_3(t)=v''(t)-\frac{n-4}{2}v'(t),\\
w_4(t)=v'''(t)-\frac{n-4}{2}v''(t),
\end{array} \right.
\end{eqnarray*}
which makes \eqref{2.2} become
\begin{equation}
\label{2.4}
\left \{ \begin{array}{l}
w'_1(t)=\frac{n-4}{2}w_1(t)+w_2(t),\\
w'_2(t)=w_3(t),\\
w'_3(t)=w_4(t),\\
w'_4(t)=C_2w_2(t)+C_3 w_3(t) +C_4 w_4(t)+ w_1^{\frac{n+4}{n-4}},
\end{array} \right.
\end{equation}
where $ C_2=\frac{2}{n-4}K_0,~ C_3=\frac{n^2}{4}>0,~C_4=-\frac{n-4}{2}.$
System \eqref{2.4} has two stationary points corresponding to $v_0$ and $v_s$:
$$
O(0,0,0,0) \;\;\; \text{and} \;\;\; P(K_0^{\frac{n-4}{8}}, -\frac{n-4}{2}K_0^{\frac{n-4}{8}}, 0,0).
$$
The linearized matrix at $O$ is
$$
M_o=\left(\begin{array}{cccc}\frac{n-4}{2}&1&0&0\\0&0&1&0\\0&0&0&1\\0&C_2&C_3&C_4\\
\end{array}
\right)
$$
and its characteristic polynomial is
$$
\lambda\mapsto\lambda^4+K_2\lambda^2+K_0.
$$
Then the eigenvalues are given by
\begin{equation}
\label{2.5}
\lambda_1=\frac{n}{2},\quad\lambda_2=-\frac{n}{2},\quad\lambda_3=\frac{n-4}{2},
\quad\lambda_4=-\frac{n-4}{2}.
\end{equation}
It is easily seen that $\lambda_1>\lambda_3>0>\lambda_4>\lambda_2$, which means that $O$ is a hyperbolic point. Moreover  both the stable and unstable manifolds are two-dimensional.

The linearized matrix at $P$ is given by
$$
M_P=\left(\begin{array}{cccc}\frac{n-4}{2}&1&0&0\\0&0&1&0\\0&0&0&1\\\frac{n+4}{n-4}K_0&C_2&C_3&C_4\\
\end{array}
\right)
$$
and the corresponding characteristic polynomial is
$$
\mu\mapsto\mu^4+K_2\mu^2-\frac{8}{n-4}K_0.
$$
Then direct computation gives out the eigenvalues
$$
\mu_1=\frac{\sqrt{n^2-4n+8+\sqrt{n^4-64n+64}}}{2},\quad\mu_2=-\frac{\sqrt{n^2-4n+8+\sqrt{n^4-64n+64}}}{2},$$
$$
\mu_3=\frac{\sqrt{n^2-4n+8-\sqrt{n^4-64n+64}}}{2},\quad\mu_4=-\frac{\sqrt{n^2-4n+8-\sqrt{n^4-64n+64}}}{2}.
$$
Note that for any $n\geq5$, $\mu_1, \mu_2 \in \mathbb{R}$, $\mu_2<0<\mu_1$ and $\mu_3, \mu_4 \not \in \mathbb{R}$, $Re\mu_3=Re\mu_4=0.$

Let $u \in C^4 (\R^n \backslash \{0\})$ be a positive entire solution of \eqref{1.1} (note that $u$ is radially symmetric) and $v$ be defined in \eqref{2.1}. Then $v(t)$ satisfies \eqref{2.2}.
In order to study the behaviors of $u$ near $0$ and $+\infty$, we need to investigate the behaviors of $v(t)$ near $\mp\infty$. At the same time, we know that \eqref{2.2} has two equilibrium points: $0$ and $K_0^{\frac{n-4}{8}}$.

\begin{lem}
\label{l2.1}
Let $v \in C^4 (\mathbb{R})$ be a positive solution of \eqref{2.2}. Assume that there exists $\theta \in [0, +\infty]$ such that $\lim_{t \to \pm\infty} v(t)=\theta$. Then $\theta \in{\{0, K^{\frac{n-4}{8}}_0\}}$.
\end{lem}

 \begin{proof}

 We only consider the case of $t \to +\infty$. The case of $t \to -\infty$ can be treated similarly. In fact, if we make the change $s=-t$, we see that $\tilde{v}(s):=v(t)$ satisfies the same equation \eqref{2.2}.

{\it Case 1}. Assume that $\theta$ is finite and $\theta \not \in{\{0, K^{\frac{n-4}{8}}_0\}}$, then $[v^{\frac{n+4}{n-4}}(t)-K_0 v(t)] \to \alpha :=\theta^{\frac{n+4}{n-4}}-K_0 \theta \neq 0$ as $t \to +\infty$. So for any $\epsilon>0$ there exists $T>0$ such that
\begin{equation}
\label{2.6}
\alpha-\epsilon\leq v^{(4)}(t)+K_2v''(t)\leq \alpha+\epsilon \quad\quad  \forall ~t\geq T.
\end{equation}
Take $\epsilon<|\alpha|$, so that $\alpha-\epsilon$ and $\alpha+\epsilon$ have the same sign and let $$\delta=\sup_{t\geq T}|v(t)-v(T)|<\infty.$$
Integrating \eqref{2.6} over $[T, t]$ for any $t\geq T$, we obtain
$$
(\alpha-\epsilon)(t-T)+C_1 \leq v^{(3)}(t)+K_2 v'(t)\leq (\alpha+\epsilon)(t-T)+C_1,
$$
where $C_1=C_1 (T)$ is a constant containing all the terms $K_2 v'(T)$ and $v^{'''}(T)$. Repeating this procedure gives for any $t>T$
\begin{eqnarray*}
&&\quad\frac{\alpha-\epsilon}{2}(t-T)^2+C_1 (t-T)-|K_2|\delta+C_2 (T)\\
&&  \leq  v''(t)\leq \frac{\alpha+\epsilon}{2}(t-T)^2+C_1 (t-T)+|K_2|\delta+C_2 (T)
\end{eqnarray*}
and
$$
\frac{\alpha-\epsilon}{6}(t-T)^3+O(t^2)\leq v'(t)\leq \frac{\alpha+\epsilon}{6}(t-T)^3+O(t^2)\quad\quad \mbox{as $t \to +\infty$}.
$$
Hence,
$$
\frac{\alpha-\epsilon}{24}(t-T)^4+O(t^3)\leq v(t)\leq \frac{\alpha+\epsilon}{24}(t-T)^4+O(t^3).
$$
This contradicts the assumption that $v(t)$ admits a finite limit $\theta$ as $t \to +\infty$.

{\it Case 2}. $\theta=+\infty$. Then there exists $T$ such that
$$
v^{(4)}(t)+K_2v''(t)=v^{\frac{n+4}{n-4}}(t)-K_0v(t)\geq \frac12v^{\frac{n+4}{n-4}}(t) \;\;\;
\forall t>T
$$
and
\begin{equation}
\label{2.7}
v^{(3)}(t)+K_2 v'(t)\geq \frac{1}{2}\int^t_T v^{\frac{n+4}{n-4}}(s) ds+C(T) \;\;\; \forall t>T.
\end{equation}
From $\lim_{t \to +\infty} v(t)=+\infty$ and \eqref{2.7} we deduce that there exists $T_1 \geq T$ such that
$$
v^{(3)}(T_1)+K_2v'(T_1)>0.
$$
Since equation \eqref{2.2} is autonomous, we may assume that $T_1=0$. Therefore we may assume
\begin{equation}
\label{2.8}
 v^{(4)}(t)+K_2v''(t)\geq \frac12 v(t)^{\frac{n+4}{n-4}} \qquad \forall t\geq 0.
\end{equation}
and
\begin{equation}
\label{2.9}
v^{(3)}(0)+K_2v'(0)=\gamma>0.
\end{equation}
Now we apply the test function method developed by Mitidieri-Pohozaev in \cite{MP}. More precisely, we can choose a nonnegative function $\phi_0\in C^{\infty}\left( [0, \infty)\right)$ satisfying $\phi_0>0$ in $[0, 2)$,
\begin{equation}
\label{2.10}
\phi_0(\tau)=\left \{ \begin{array}{ll}
1, \;\;\; &\mbox{for $\tau \in [0,1]$},\\
0, \;\;\; &\mbox{for $\tau \geq 2$},\\
\int_0^2\frac{|\phi^{(i)}_0(\tau)|^2}{\phi_0(\tau)} d\tau:=A_i<\infty, \;\;\;& \forall i\in\mathbb{N}.
\end{array} \right.
\end{equation}
Let $T_1:=2 T$, multiplying inequality \eqref{2.8} by $\phi(t):=\phi_0 (\frac{t}{T})$ and integrating by parts, we obtain
$$
\int^{T_1}_0 v^{(4)}(t)\phi(t)dt+K_2\int^{T_1}_0 v''(t)\phi(t) dt\geq\frac{1}{2}\int^{T_1}_0 v(t)^{\frac{n+4}{n-4}}\phi(t) dt.
$$
Hence
\begin{equation}
\label{2.11}
\int^{T_1}_0 \phi^{(4)}(t)v(t)dt+K_2\int^{T_1}_0 \phi''(t)v(t) dt\geq\frac{1}{2}\int^{T_1}_0 v(t)^{\frac{n+4}{n-4}}\phi(t) dt +\gamma.
\end{equation}
Note that $v^{(3)}(0)+K_2v'(0)=\gamma>0$ and $\phi(T_1)=\phi'(T_1)=\phi''(T_1)=\phi^{(3)}(T_1)=0$.
By Young's inequality, for any $\varepsilon>0, \;\; \exists \ \ C_{\varepsilon}>0$ such that
$$
v(t)|\phi^{(i)}(t)| \leq \varepsilon v^{\frac{n+4}{n-4}}\phi+C_{\varepsilon}\frac{|\phi^{(i)}|^{\frac{n+4}{8}}}{\phi^{\frac{n-4}{8}}}.
$$
Then provided  $\varepsilon$  sufficiently small, \eqref{2.11} yields
$$
{\tilde C} \Big(\int^{T_1}_0(\frac{|\phi^{(4)}|^{\frac{n+4}{8}}}{\phi^{\frac{n-4}{8}}}
+\frac{|\phi''|^{\frac{n+4}{8}}}{\phi^{\frac{n-4}{8}}})dt \Big)\geq \frac{1}{4}\int^{T_1}_0 v^{\frac{n+4}{n-4}}\phi dt +\gamma
$$
with a fixed constant ${\tilde C}>0$. Using  \eqref{2.10}, we get
$$
{\tilde C}_1 (A_4 T^{-\frac{n+2}{2}}+A_2 T^{-\frac{n}{4}})\geq\frac{1}{4}\int^T_0 v^{\frac{n+4}{n-4}} dt
$$
with a fixed constant ${\tilde C}_1>0$. Sending $T$ to $\infty$, we observe that the left-hand side of the above inequality goes to $0$, while the right-hand side of the above inequality goes to $+\infty$ since $v(t) \to +\infty$ as $t \to +\infty$, which is a contradiction. The contradictions in cases 1 and 2 complete the proof of this lemma.

\end{proof}

\begin{lem}
\label{l2.2}
Assume that $v \in C^4 (\mathbb{R})$ is a positive solution of equation \eqref{2.2} and satisfies $\lim_{t \to \pm\infty}v(t)=0$. Then for all $k \in \mathbb{N}, k\ge 1$,  $\lim_{t \to \pm\infty} v^{(k)}(t)=0$.
\end{lem}

\begin{proof}

 We only give the proof for the case $t \to +\infty$. By the assumption, $v(t)<K_0^{\frac{n-4}{8}}$  for $t$ large enough. Then
$$v^{(4)}(t)+K_2v''(t)= (v^{\frac{8}{n-4}}-K_0) v(t)<0.$$
Let $H(t):=v''(t)+K_2 v(t)$. Then $H''(t)<0$ and $H$ is concave near $\infty$. So $\lim_{t\to\infty} H(t)=b\in \mathbb{R} \cup \{\pm \infty\}$, hence $\lim_{t\to\infty}v''(t)=b$.
If $b=\pm \infty$, we have $\lim_{t \to +\infty} v'(t)=\pm \infty$ respectively. These reach contradictions to $v(t) \to 0$ as $t \to +\infty$. Hence $b \in \mathbb{R}$ and
$\lim_{t \to +\infty} v''(t)=b$. We claim $b=0$. On the contrary, if $b>0$, we have $v''(t)\geq \frac{b}{2}$ and hence $v'(t)\geq \frac{b}{4}t$ for $t$ sufficiently large. This is impossible in view of $\lim_{t \to +\infty} v(t)=0$. If $b<0$, the same argument implies that $v'(t)\leq \frac{b}{4} t$ for $t$ sufficiently large and a contradiction to $\lim_{t \to +\infty} v(t)=0$. Therefore,
$$\lim_{t \to +\infty} v''(t)=b=0.$$

Next, we show that the limit of $v'(t)$ exists as $t \to +\infty$. For $t$ large enough, there exists $\xi\in[t, t+1]$ such that
$$v(t+1)-v(t)=v'(t)+ \frac12v''(\xi).$$
As  $t\rightarrow +\infty$, it is obvious that $\xi \to +\infty$, $v(t+1) \to 0$, $v(t) \to 0$ and $v''(\xi) \to 0$. Hence
$$\lim_{t \to +\infty}v'(t)=0.$$
Now  the fact $\lim_{t \to +\infty} H(t)=0$ and the concavity of $H$ imply $\lim_{t\to+\infty}H'(t)=0$, hence $\lim_{t \to +\infty}v^{(3)}(t)=0$. The equation \eqref{2.2} yields $$\lim_{t \to +\infty} v^{(k)}(t)=0, \qquad \mbox{ for}\quad k \geq 4.$$
\end{proof}

\begin{lem}
\label{l2.3}
Assume that $v \in C^4 (\mathbb{R})$ is a positive solution of \eqref{2.2}.  If  $v(t)$ is eventually  monotone and satisfies $\lim_{t \to \pm\infty} v(t)=K_0^{\frac{n-4}{8}}$, then for all $k\geq 1$, $\lim_{t \to \pm\infty} v^{(k)}(t)=0.$
\end{lem}

\begin{proof}
We only consider the case of $t \to +\infty$. Since $v(t)$ is eventually monotone as $t \to +\infty$ and satisfies $\lim_{t \to +\infty}v(t)=K_0^{\frac{n-4}{8}}$, there are two cases:
(i) $v(t)$ is eventually decreasing as $t \to +\infty$, (ii) $v(t)$ is eventually increasing as $t \to +\infty$.

The proof is almost the same, hence here we just give the proof of the first case. In this case,  there is $T \gg 1$ such that $v(t)$ is decreasing and
$v(t)\geq K_0^{\frac{n-4}{8}}$ for $t>T$. It is easily seen from \eqref{2.2} that
$$
v^{(4)}(t)+K_2v''(t)= (v^{\frac{8}{n-4}}-K_0) v(t) \geq 0 \;\;\; \forall t>T.
$$
Let $H(t):=v''(t)+K_2 v(t)$. Then $H''(t) \geq 0$ for $t>T$ and $H$ is convex near $\infty$. So $\lim_{t\to \infty} H(t)=a \in \mathbb{R} \cup \{\pm \infty\}$. Noticing that $v(t) \to K_0^{\frac{n-4}{8}}$ as $t \to \infty$, we see that
$$
\lim_{t \to +\infty} v''(t)=b\in \mathbb{R} \cup \{\pm \infty\},
$$
where $b=a-K_2 K_0^{\frac{n-4}{8}}$.
If $b=\pm \infty$, we easily see $v'(t)=\pm \infty$ as $t \to +\infty$ respectively and hence $v(t) \to \pm \infty$ as $t \to +\infty$. These contradict to $v(t) \to K_0^{\frac{n-4}{8}}$ as $t \to +\infty$. So
$b \in \mathbb{R}$. We claim $b=0$. On the contrary, if $b>0$, we have $v''(t)\geq \frac{b}{2}$
for $t$ sufficiently large and hence $v'(t)\geq \frac{b}{4}t$ for $t$ sufficiently large. This is impossible since $v(t) \to K_0^{\frac{n-4}{8}}$ as $t \to +\infty$. If $b<0$, we can obtain $v'(t)\leq \frac{b}{4} t$ for $t$ sufficiently large. This implies $v(t)\to -\infty$ as $t \to +\infty$. A contradiction to $v(t) \to K_0^{\frac{n-4}{8}}$ as $t \to +\infty$. Therefore, $$\lim_{t \to +\infty} v''(t)=b=0,\quad\quad\lim_{t\to+\infty}H(t)=K_2 K_0^{\frac{n-4}{8}}.$$
Next, we show that the limit of $v'(t)$ exists as $t \to +\infty$. For $t$ large enough, there exists $\xi\in[t, t+1]$ such that
 $$
 v(t+1)-v(t)=v'(t)+\frac{1}{2} v''(\xi).
 $$
Let $t \to +\infty$, we see that $\xi \to +\infty$, $v(t+1) \to K_0^{\frac{n-4}{8}}, v(t) \to K_0^{\frac{n-4}{8}}$ and $v''(\xi) \to 0$. Hence $\lim_{t \to +\infty}v'(t)=0$.
Now  the fact $\lim_{t \to +\infty} H(t)=K_2 K_0^{\frac{n-4}{8}}$ and the convexity of $H$ imply $\lim_{t\to+\infty}H'(t)=0$, so $\lim_{t\to+\infty} v^{(3)}(t)=0$. Equation \eqref{2.2} yields $\lim_{t \to +\infty} v^{(k)}(t)=0$ for $k \geq 4$. This completes the proof.
\end{proof}

To prove a non-constant solution $v$ of \eqref{2.2} to be oscillatory, an useful energy function of $v$ is introduced, which helps to exclude the possibility of $v$ to be monotone near $t=\pm\infty$. For any positive solution $v \in C^4 (\mathbb{R})$ of \eqref{2.2}, we define the energy function
$$
\tilde{E}_v(t)=E_v(t)-v'(t)v^{(3)}(t),
$$
where
$$
E_v (t)=\frac{1}{2} v''^2(t)-\frac{K_2}{2} v'^2(t)-\frac{K_0}{2} v^2+\frac{n-4}{2n}v^{\frac{2n}{n-4}}(t).
$$

\begin{lem}
\label{l2.4}
Assume that $v \in C^4 (\mathbb{R})$ is a positive solution of \eqref{2.2}, then there exists
$\mu \in \mathbb{R}$ such that $\tilde{E}_v(t)\equiv \mu$ for $t \in \mathbb{R}$.
\end{lem}
\begin{proof}

 It is easily obtained from \eqref{2.2} that
\begin{equation}
\label{2.12}
\begin{aligned}
E'_v(t)&=v''(t)v^{(3)}(t)-K_2v'(t)v''(t)-K_0 v(t)v'(t)+v^{\frac{n+4}{n-4}}(t)v'(t)\\
&=(v^{\frac{n+4}{n-4}}(t)-K_2v''(t)-K_0 v(t))v'(t)+v''(t)v^{(3)}(t)\\
&=v^{(4)}(t)v'(t)+v''(t)v^{(3)}(t)=(v'(t)v^{(3)}(t))'.
\end{aligned}
\end{equation}
Then
$$
\tilde{E}'_v(t)=E'_v(t)-(v'(t)v^{(3)}(t))'\equiv 0, \;\; \mbox{for $t \in \mathbb{R}$}.
$$
Hence there exists $\mu \in \mathbb{R}$ such that
$\tilde{E}_v(t)\equiv \mu$ for $t\in \mathbb{R}$.
\end{proof}

We now have the following key lemma.
\begin{lem}
\label{l2.5}
Assume that $ v \in C^4 $ is a positive entire solution of  \eqref{2.2}.
Then, the equation: $v'(t)=0$ admits infinitely many roots in $\mathbb{R}$.
\end{lem}

 \begin{proof}
 We only consider the non-constant case. Suppose by contradiction that $v'(t)=0$ has finite roots in $\mathbb{R}$, then $v(t)$ is monotone for large $|t|$.  Lemma \ref{l2.1} leads to $v(t) \to 0$ or $K_0^{\frac{n-4}{8}}$ as $t \to \pm\infty$. Denoting $l:=K_0^{\frac{n-4}{8}}$ in the following for simplicity.

{\it Case 1.}  $v(t) \to 0$ as $t \to -\infty$, $v(t) \to l$ as $t \to +\infty$.
It follows from Lemma \ref{l2.4} that
$${\tilde E}_v (t) \equiv \mu \;\;\; \mbox{for $t \in \mathbb{R}$}.$$
Therefore,
\begin{equation}
\label{2.13}
{\tilde E}_v (-\infty)={\tilde E}_v (+\infty).
\end{equation}
On the other hand, we see from Lemmas \ref{l2.2} and \ref{l2.3} that
$${\tilde E}_v (-\infty)=0, \;\;\; {\tilde E}_v (+\infty)=G(l)=-\frac{2}{n}K_0^{\frac{n}{4}}<0$$
where
\begin{equation}
\label{2.13-1}
G(s)=-\frac{K_0}{2} s^2+\frac{n-4}{2n} s^{\frac{2n}{n-4}}.
\end{equation}
This contradicts to \eqref{2.13}.

{\it Case 2.} $v(t) \to 0$ as $t \to +\infty$, $v(t) \to l$ as $t \to -\infty$.
We can also derive a contradiction  by arguments similar to those in the case 1.

{\it Case 3.} $v(t) \to l$ as $t \to \pm\infty$. Since $v(t) \not \equiv l$, $v(t)$ admits
a global maximum or minimum point $t_0 \in \mathbb{R}$ and $v(t_0) \neq l$. It follows from \eqref{2.12} and Lemma \ref{l2.3} that
$$
E_v(+\infty)-E_v(t_0)=\int^{+\infty}_{t_0} E_v'(s) ds= v^{(3)}v'\mid^{+\infty}_{t_0}=-v^{(3)}(t_0)v'(t_0)=0.
$$
Hence
\begin{equation}
\label{2.14}
E_v(+\infty)=E_v(t_0).
\end{equation}
Obviously,
$$E_v(+\infty)=G(l),\qquad \qquad E_v(t_0)=\frac{1}{2}v''^2(t_0)+G(v(t_0)).$$
A simple calculation implies
$$G(l)=\min_{s \in [0, \infty)} G(s) \; \Big(=-\frac{2}{n}K_0^{\frac{n}{4}}<0 \Big).$$
Since $v(t_0) \neq l$, no matter $v(t_0)<l$ or $v(t_0)>l$, we have $G(v(t_0))>G(l)$. Therefore,
\begin{equation}
\label{2.15}
E_v(t_0)>E_v (+\infty).
\end{equation}
This contradicts to \eqref{2.14}.

{\it Case 4.} $v(t) \to 0$ as $t \to \pm\infty$.
 If we make the transformation: $s=-t$ and
$w(s)=v(t)$, we see that $w(s)$ satisfies the same equation \eqref{2.2}. Moreover,
$w(s) \to 0$ as $s \to +\infty$.
 It follows from the ODE theory that there is $S \gg 1$ such that
\begin{eqnarray*}
& & w(s)=A_1 e^{-\frac{n}{2} s}+A_2 e^{-\frac{n-4}{2} s}+M_3 e^{\frac{n-4}{2}s}+M_4 e^{\frac{n}{2} s}\\
& &\;\;\;\;+B_1 \int_S^s e^{-\frac{n}{2} (s-t)} g(w(t)) dt+B_2 \int_S^s e^{-\frac{n-4}{2} (s-t)} g(w(t)) dt\\
& &\;\;\;\;-B_3 \int_s^\infty e^{\frac{n-4}{2} (s-t)} g(w(t)) dt-B_4 \int_s^\infty e^{\frac{n}{2} (s-t)} g(w(t)) dt,
\end{eqnarray*}
where $g(w(t))=w^{\frac{n+4}{n-4}} (t)$, the constants $A_1,A_2, M_3, M_4$ depend on $S$,
the constants $B_1, B_2, B_3, B_4$ are independent of $S$. Note that the four eigenvalues are given in \eqref{2.5}.
Since $w(s) \to 0$ as $s \to +\infty$, we see that $M_3=M_4=0$. Arguments similar to those in the proof of Theorem 3.1 of \cite{Guo} and Proposition 3.1 of \cite{GHZ} imply that
$$w(s)=O \Big(e^{-\frac{n-4}{2} s} \Big) \qquad \mbox{for $s$ near $+\infty$}.$$
This implies that
$$v(t)=O(e^{\frac{n-4}{2}t}) \qquad  \mbox{for $t$ near $-\infty$}.$$
Since $u(r)=e^{-\frac{n-4}{2} t} v(t)$, we obtain
$$u(r)=O(1) \qquad \mbox{for $r$ near 0}$$
and $0$ is a removable singularity point of $u$. This contradicts our assumption that $0$
is a non-removable singularity of $u(r)$. All the contradictions derived above imply that
the equation $v'(t)=0$ admits infinitely many roots in $\mathbb{R}$.

\end{proof}

We present the precise behavior of $v (t)$ at its extremal points in the following lemma.

\begin{lem}
\label{l2.6}
Assume that $v \in C^4 (\mathbb{R})$ is a non-constant positive solution of equation \eqref{2.2} and ${\{t_i}\}_{i=1}^{\infty}$ and ${\{s_i}\}_{i=1}^{\infty}$ are the local maximum and minimum points of $v$ respectively. Then $v(t_i)>l$, $v(s_i)<l$ for any $i$.
\end{lem}

 \begin{proof}
 Suppose that there exists a maximum point $t_0$ of $v$ such that $v(t_0)\leq l$ and $v'(t_0)=0$, $v''(t_0)\leq 0$.  It follows from equation \eqref{2.2} that either
 $$
 \mbox{(i)} \qquad v(t_0)=l, \qquad v''(t_0)=0, \qquad v^{(4)}(t_0)=0
 $$
 or
$$ \mbox{(ii)}\qquad  v^{(4)}(t_0)=-K_2 v''(t_0)+ v^{\frac{n+4}{n-4}}(t_0)-K_0 v  (t_0) < 0.$$

For the case (i), we claim that $v^{(3)}(t_0)=0$. Then by the uniqueness we get that $v(t)\equiv l$, which is impossible. Now we prove the claim. We may assume $v^{(3)}(t_0)> 0$, otherwise we consider $w(s)=v(t)$ with $s=-t$. By the elementary analysis we deduce $\exists \ \delta >0$ such that
\[
v''(t)>0, \qquad v'(t)>0 \qquad \mbox{in} \quad (t_0, t_0+\delta).
\]
Contradict to the local maximum point $t_0$ of $v(t)$.

For the case (ii),  $v^{(4)} (t)<0$ for $t \in (t_0, t_0+\epsilon)$ for some small $\epsilon>0$. In this case we assume $v^{(3)}(t_0) \leq 0$. Hence  $v'(t)<0$ and $v^{(3)}(t)<0$ for $t\in (t_0, t_0+\epsilon)$.
Now let
$$
t_1=\sup \left\{\tilde{t}>t_0: v'(t)<0, \;\; v^{(3)}(t)<0, \;\;\; t\in(t_0£¬~\tilde{t})\right\}.
$$

Then $t_1$ is finite by  oscillation of $v(t)$ and either $v'(t_1)=0$ or $v'''(t_1)=0$ because of continuity. Thus $v'(t_1)v^{(3)}(t_1)=0$. Moreover, $v''$ is decreasing in $(t_0,t_1)$. Since $v''(t_0) \leq 0$, we see that $v''^2(t)$ is increasing for $t \in (t_0, t_1)$. It follows from Lemma \ref{l2.4} that
$$\frac{1}{2} v''^2 (t_1) -\frac{K_2}{2} v'^2(t_1)+G(v(t_1))=\frac{1}{2}v''^2(t_0)+G(v(t_0)),$$
since $v'(t_0)=0$ and $v'(t_1)v^{(3)}(t_1)=0$. Thus,
\begin{equation}
\label{2.16}
G(v(t_0))-G(v(t_1))=\frac{1}{2}(v''^2(t_1)-v''^2(t_0))-\frac{K_2}{2}v'^2(t_1)>0,
\end{equation}
where $G(s)$ is given in \eqref{2.13-1}.
On the other hand, since $v(t_1)<v(t_0) \leq l$, we easily see $G(v(t_1))>G(v(t_0))$. This contradicts to \eqref{2.16}. Hence $v(t_i)>l$ if $t_i$ is a local maximum point of $v$. Similarly, we can obtain $v(s_i)<l$ if $s_i$ is a local minimum point of $v$. The proof is completed.
\end{proof}

{\bf Proof of Theorem 1.2}\\

It follows from \cite{L} that $u$ is radially symmetric with respect to $x=0$. By Lemmas \ref{l2.5} and \ref{l2.6}, we see that  the function $v(t)-l$ changes signs infinitely many times
in $\mathbb{R}$ with $v(t):=e^{\frac{n-4}{2} t} u(e^t)$. It implies that $u(r)-K^{\frac{n-4}{8}}_0 r^{-\frac{n-4}{2}}$ changes signs infinitely many times in $(0,\infty)$. The proof of Theorem  1.2 is completed.
\qed

\section{Existence of positive periodic solutions of \eqref{2.2}: Proof of Theorem \ref{t1.2}}
\setcounter{equation}{0}

In this section, we will see that equation \eqref{2.2} admits positive periodic solutions
and present the proof of Theorem \ref{t1.2}. In the following, we denote $C$ a positive constant which may be changed from one line to another line.

Let $G(s)$ be given in \eqref{2.13-1} and $F(s)=-G(s)$. It is easily seen that $F(0)=F(L)=0$
and $F(s)>0$ for $s \in (0, L)$, $F(s)<0$ for $s>L$, where
$$L=\Big(\frac{n}{n-4} \Big)^{\frac{n-4}{8}} K_0^{\frac{n-4}{8}}>l:=K_0^{\frac{n-4}{8}}.$$

Let $-\infty<t_1<t_2<\infty$ and
$$H=\left\{v \in H^{2}(t_1, t_2): \; v' \in H_0^{1}(t_1, t_2)\right\}$$
with the scalar product
$$(u,v)_H:=\int^{t_2}_{t_1} \left[u'' (t) v''(t)+u' (t) v' (t)+u (t) v (t)\right] dt$$
and
$$\|u\|_H^2=(u,u)_H.$$
Then $H$ is a Hilbert space. For simplicity, if we denote
$$\|u\|^2_{K_2,K_0}:=\int^{t_2}_{t_1} [u''^2 (t)-{K_2}u'^2 (t)+K_0 u^2 (t)] dt,$$
we see that
$$\|u\|_H^2 \leq \|u\|^2_{K_2,K_0}$$
since
$$K_2=-\frac{n^2-4n+8}{2}<-4, \;\;\; K_0=\frac{n^2(n-4)^2}{16}>1$$
for $n \geq 5$.

Define the functional
\begin{equation}
\label{3.1}
J (v)=\int^{t_2}_{t_1} \Big[\frac{v''^2(t)}{2}-K_2\frac{v'^2(t)}{2}+F(v (t)) \Big]dt, \;\; v \in H.
\end{equation}

\begin{lem}
\label{l3.1}
The functional $J$ satisfies $(PS)_c$ condition on the sequence $\{v_j\}_{j=1}^\infty \subset H$
with $0 \leq v_j \leq L$ in $[t_1, t_2]$ for each $j$.
\end{lem}

\begin{proof}
 Let $\{v_j\}_{j=1}^\infty \subset H$ with $0 \leq v_j \leq L$ in $[t_1, t_2]$ for each $j$ and
$$J(v_j) \to c, \;\;\; (J'(v_j), v_j) \to 0 \;\;\mbox{as $j \to \infty$}.$$
Since
$$J(v_j)=\frac{1}{2}\|v_j\|^2_{K_2,K_0}-\frac{n-4}{2n}\int^{t_2}_{t_1}v_j^{\frac{2n}{n-4}} (t)dt\leq C$$
and
$$(J' (v_j), v_j)=\|v_j\|^2_{K_2,K_0}- \int^{t_2}_{t_1}v_j^{\frac{2n}{n-4}} (t)dt \to 0,$$
we get, $\|v_j\|_{K_2,K_0} \leq C$ and $\|v_j\|_H \leq C$. Observe that $D J$ is of the form $Id+K$ with $K$ compact and therefore $\{v_j\}_{j=1}^\infty$ (up to some subsequence, still denoted by $\{v_j\}_{j=1}^\infty$) is convergent to $ v$ in $H$. The emdedding and Arzela-Ascoli theorem imply
that $v_j \to v$ in $C^0 ([t_1,t_2])$ and $0 \leq v \leq L$ in $[t_1, t_2]$.

\end{proof}
\begin{lem}
\label{l3.2}
For any fixed $-\infty<t_1<t_2<\infty$, $J$ admits a mountain pass critical point $v_0 \in H$
with $0 \leq v_0 \leq L$ in $[t_1, t_2]$, and
$$J( v_0)=\inf_{\gamma\in \Gamma }\sup_{t\in[0,1]} J(\gamma_t)>0,$$
where
$$\Gamma:=\{\gamma \in C^0([0,1],H): \; 0 \leq \gamma_t \leq L \;\; \mbox{in $[t_1, t_2]$ for $t \in [0,1]$, $\gamma_0 \equiv 0$, $\gamma_1 \equiv L$} \}.$$
\end{lem}

\begin{proof}
 We show that the functional $J$ has the mountain-pass geometry.

The embedding implies $\|v\|_\infty \leq C(t_1,t_2) \|v\|_H$ for some $C(t_1,t_2)>0$.

Let $\epsilon>0$ be sufficiently small such that $F(t)>F(0)=0$ for $t \in(0, \epsilon)$ and $F(t)>F(L)=0$ for $t \in (L-\epsilon, L)$. Since $-K_2>0$,  $0 \leq v \leq L$ in $[t_1, t_2]$ and $\|v\|_H \leq \frac{\epsilon}{C(t_1,t_2)}$, it holds
\begin{equation}
\label{3.2}
J(v) \geq \int^{t_2}_{t_1} F (v(t)) dt\geq 0.
\end{equation}
We claim that the mountain pass geometry holds with $\rho=\frac{\epsilon}{C(t_1,t_2)}$, i.e. $\inf_{\|v\|_{H=\rho}} J(v)>0$. Suppose not, then there exists ${\{v_j}\}_{j=1}^\infty$ with $0 \leq v_j \leq L$ in $ [t_1, t_2]$  such that
\begin{equation}
\label{3.3}
\|v_j\|_H \equiv \rho \;\;\; \text{and} \;\;\; J (v_j)\leq \frac{1}{j}, \;\;\; \forall j\geq 1.
\end{equation}
We see
$$v_j\rightharpoonup v \;\;\; \text{in} \;\; H \;\;\;\; \text{and} \;\;\; v_j \to v
\;\;\; \text{in}\;\;\; L^2([t_1,t_2]).$$
We also have $\int_{t_1}^{t_2} F (v_j(t)) dt \to 0$ thanks to \eqref{3.2}. Since $\|v_j\|_{\infty} \leq \epsilon$, the dominated convergence gives $v \equiv 0$ in $[t_1, t_2]$. Moreover, by \eqref{3.3}
$$
\int_{t_1}^{t_2} \Big[\frac{v''^2_j(t)}{2}-K_2\frac{v'^2_j(t)}{2} \Big] dt \leq \frac{1}{j}\rightarrow 0,
$$
which implies $v_j \to 0$ strongly in $H$ and contradicts to $\|v_j\|_H \equiv \rho>0$.
Hence there exists $\delta:=\delta (\rho)>0$ such that
$$J(v)\geq \delta>0 \;\; \mbox{for all $v \in S_\rho$},$$
where
$$S_\rho=\left\{ v \in H: \; \|v\|_H=\rho, \;\; 0 \leq v \leq L \;\; \mbox{in $[t_1, t_2]$}\right\}.$$
Since $J(0)=0=J(L)$ and $\|L\|_H=L (t_2-t_1)^{\frac{1}{2}}>\rho$ (by choosing $\epsilon$ sufficiently small),  mountain-pass Lemma (see \cite{Chang}, Theorem 4.8.5) shows that
$$c:=\inf_{\gamma \in \Gamma} \sup_{t \in[0,1]} J (\gamma_t)>0$$
is a critical value of $J$. That is, there exists a critical point $v_0 \in H$ of $J$
with $0 \leq v_0 \leq L$ in $[t_1,t_2]$  such that $J(v_0)=c$.
\end{proof}

\begin{lem}
\label{l3.3}
Let $T:=t_2-t_1>0$. If $T$ is suitably large, the critical point $v_0$ obtained in Lemma \ref{l3.2} is nonconstant.
\end{lem}

\begin{proof}
 To prove this lemma, we use some ideas similar to those in \cite{MS}. Suppose on the contrary that $v_0\equiv \kappa$ in $[t_1, t_2]$. By Lemma \ref{l3.2}, we have
\begin{equation}
\label{3.4}
0<c:=J(v_0)=F(\kappa)(t_2-t_1),
\end{equation}
which implies $F(\kappa)>0$. To get a contradiction with Lemma \ref{l3.2} and \eqref{3.4}, we need to build a curve $\tilde{\gamma} \in \Gamma$ such that
\begin{equation}
\label{3.5}
\sup_{t \in [0, 1]} J (\tilde{\gamma}_t)<F (\kappa)(t_2-t_1).
\end{equation}

Suppose $t_2-t_1>2$ and fix an auxiliary function $\phi\in C^{\infty}([t_1,t_1+1])$ with
$\phi (t) \geq 0$ for $t \in [t_1, t_1+1]$ and
$$\phi(t_1)=1, \;\; \phi'(t_1)=0, \;\; \phi(t_1+1)=0, \;\; \phi'(t_1+1)=0, \;\; \max_{t \in [t_1, t_1+1]} \phi (t)=1.$$
Consider the curves $\tilde{\gamma}^i\in C^0([0, 1]; H^2 (t_1,t_1+1)), i=1,2$ defined as
$$\tilde{\gamma}^1_t(x)=Lt\phi(x), \;\;\; \tilde{\gamma}^2_t(x)=L((1-t)\phi(x)+t).$$
Then they have the following properties:
\begin{equation}
\label{3.6}
\begin{cases}
\tilde{\gamma}^1_t(t_1+1)=0,~~\tilde{\gamma}^2_t(t_1)=L,~~~~\quad&\forall\ t\in[0,1],\\
(\tilde{\gamma}^i_t)'(t_1)=(\tilde{\gamma}^i_t)'(t_1+1)=0,\quad &\text{for} \ ~i=1,~2~\text~\forall t\in[0,1],\\
\tilde{\gamma}^1_1(x)=\tilde{\gamma}^2_0(x), \quad\quad &\forall\  x\in[t_1,~t_1+1],\\
\tilde{\gamma}^1_0(x)\equiv 0,~~\tilde{\gamma}^2_1(x)\equiv L, \quad\quad &\forall\ x\in[t_1,~t_1+1].
\end{cases}
\end{equation}
Now let
$$
E_0=\sup_{t\in[0,1]} \max \{J (\tilde{\gamma}^1_t), \;\; J (\tilde{\gamma}^2_t)\}
$$
and here $J$ is considered to be the energy function with integrals defined on $[t_1, t_1+1]$.

Define the curve
$$
\gamma^1_t(x):=
\begin{cases}
\tilde{\gamma}^1_t(x)\quad\quad &\text{for $x\in[t_1,t_1+1]$},\\
0\quad\quad\quad &\text{for $x\in[t_1+1, t_2]$}.
\end{cases}
$$
Notice that for all $t\in[0,1]$, it holds
\begin{equation}
\label{3.7}
J(\gamma^1_t)=\int^{t_1+1}_{t_1}
\Big[\frac{|(\tilde{\gamma}^1_t)'' (x)|^2}{2}-\frac{K_2}{2}|(\tilde{\gamma}^1_t)' (x)|^2
+F (\tilde{\gamma}^1_t (x)) \Big] dx\leq E_0.
\end{equation}
Furthermore, define the curve
$$
\gamma^2_t(x):=
\begin{cases}
L\quad &\text{for $x\in[t_1,~t\hat{T}+t_1]$},\\
\tilde{\gamma}^1_1(t_1+x-[t\hat{T}+t_1])\quad &\text{for $x\in[t\hat{T}+t_1, t\hat{T}+t_1+1]$},\\
0\quad  &\text{for $x\in[t\hat{T}+t_1+1, t_2]$}.
\end{cases}
$$
where $\hat{T}=t_2-t_1-1$.
By translation invariance, it holds
\begin{equation}
\label{3.8}
J(\gamma ^2_t)\leq E_0,
\end{equation}
for all $t\in[0, 1]$. Finally let
$$\gamma^3_t (x):=
\begin{cases}
L\quad&\text{for $t_1 \leq x\leq t_2-1$},\\
\tilde{\gamma}^2_t(t_1+x-(t_2-1))\quad&\text{for $t_2-1\leq x\leq t_2$},
\end{cases}
$$
 then
\begin{equation}
\label{3.9}
J (\gamma ^3_t)\leq E_0,
\end{equation}
for all $t\in[0,1]$. Using the properties in \eqref{3.6} one can check that  $\gamma^i$ belong to $C^0([0,1], H)$ with $0 \leq \gamma^i_t \leq L$ for $t \in [0,1]$, and  they can be concatenated to form a $\tilde{\gamma}\in \Gamma$. By the inequalities \eqref{3.7}-\eqref{3.9}, it holds
$$
\sup_{t\in[0,1]} J(\tilde{\gamma}_t)\leq E_0.
$$
Since ${F}(\kappa)>0$ the inequality
$$E_0<{F}(\kappa)(t_2-t_1)$$
holds for suitably large $T:=t_2-t_1$, which gives the claim \eqref{3.5}. This
completes the proof.
\end{proof}

Next, we will show that equation \eqref{2.2} admits a positive periodic solution.

\begin{lem}
\label{l3.4}
Let $v_0$ be given in Lemma \ref{l3.2}. Then equation \eqref{2.2} admits a nontrivial nonnegative periodic solution $v \in C^4 (\mathbb{R})$, which is the extension of $v_0$ in $\mathbb{R}$.
\end{lem}

\begin{proof}
 We see that $v_0 \in H$  satisfies equation \eqref{2.2} in $[t_1,t_2]$. We easily know from the embeddings that $v_0 \in C^4([t_1,t_2])$. Since  $v'_0(t_1)=v'_0(t_2)=0$, it suffices to show that $v^{(3)}_0(t_1)=v^{(3)}_0(t_2)=0$. Integrating by parts, we get for any $\phi\in H$
 $$
\begin{aligned}
0&=\int^{t_2}_{t_1} \left[v''_0\phi''-K_2 v'_0\phi'+F'(v_0) \phi\right] dt=\int^{t_2}_{t_1} \left[-v_0^{(3)}\phi'+K_2 v_0''\phi+ F'(v_0)\phi\right]dt\\
&=\int^{t_2}_{t_1} \left[v_0^{(4)}\phi+K_2 v_0''\phi+ F'(v_0)\phi\right]dt-v_0^{(3)}\phi|^{t_2}_{t_1}=-v_0^{(3)}\phi\big|^{t_2}_{t_1}.
\end{aligned}
$$
 Clearly we can choose arbitrarily  values at $t_1$ and $t_2$ for $\phi\in H$ and obtain $v_0^{(3)}(t_1)=v_0^{(3)}(t_2)=0$. In conclusion, $v_0^{(4)}(t)+K_2 v_0''(t)+K_0 v_0=v_0^{\frac{n+4}{n-4}}$
in $(t_1, t_2)$ and $v'_0(t_1)=v'_0(t_2)=v^{(3)}_0(t_1)=v^{(3)}_0(t_2)=0$. Define ${\tilde v}_0 (t)=v_0(2t_2-t)$ for $t \in [t_2, 2t_2-t_1]$, we see that ${\tilde v}_0 (t)$ satisfies
${\tilde v}_0^{(4)}(t)+K_2 {\tilde v}_0''(t)+K_0 {\tilde v}_0={\tilde v}_0^{\frac{n+4}{n-4}}$
in $(t_2, 2t_2-t_1)$ and $\tilde{v}_0^{(k)}(t_2)=v_0^{(k)}(t_2), \tilde{v}_0^{(k)}(2t_2-t_1)=v_0^{(k)}(t_1)$ for $k=0,1,2,3$. It follows from ODE theory that we can obtain a $2T:=2(t_2-t_1)$-periodic solution $v(t)$ for equation \eqref{2.2}.
\end{proof}

{\bf Proof of Theorem \ref{t1.2}}

To prove Theorem \ref{t1.2}, we only need to show that $v \in C^4 (\mathbb{R})$ obtained in Lemma \ref{l3.4} is actually  positive. On the contrary, $\min_{t \in \mathbb{R}} v (t)=0$ and there is a sequence $\{s_i\}_{-\infty}^\infty$ such that $v (s_i)=0$. Then $u (r)=r^{-\frac{n-4}{2}} v (\ln r)$ is a nonnegative radial solution to \eqref{1.1} and
there is a sequence $\{r_i\}_{-\infty}^\infty$ with $r_i=e^{s_i}$ such that
$u (r_i)=0$. We also know $u \in C^4 (\R^n \backslash \{0\})$ and
\begin{equation}
\label{3.10}
u (r) \leq L r^{-\frac{n-4}{2}},  \qquad  \forall \ r>0.
\end{equation}

Since $u$ satisfies
$$\Delta^2 u=u^{\frac{n+4}{n-4}} \qquad \mbox{in} \quad \R^n \backslash \{0\},$$
we see that, for any $r>0$ sufficiently small,
\begin{equation}
\label{3.11}
0\leq|S^{n-1}| r^{n-1} \frac{\partial \Delta u}{\partial r} (r)=\int_{B_r} u^{\frac{n+4}{n-4}} dx,
\end{equation}
where $ S^{n-1}=\partial B_1$ and $B_r=\{x \in \R^n: \; |x|<r\}$.
It follows from \eqref{3.10} that
$$\int_{B_r} u^{\frac{n+4}{n-4}} dx \leq L |S^{n-1}| \int_0^r s^{n-1-\frac{n+4}{2}} ds \leq Cr^{\frac{n-4}2},$$
where $C>0$ is independent of $r$.  We see from \eqref{3.11} that
\begin{equation}
\label{3.12}
\lim_{r \to 0^+} r^{n-1} (\Delta u)'(r)=0.
\end{equation}
Since $u$ satisfies an equation of the radial form:
$$(r^{n-1} (\Delta u)'(r))'=r^{n-1} u^{\frac{n+4}{n-4}} (r) \;\; \mbox{for $r \in (0, \infty)$},$$
 by \eqref{3.12} we can deduce
$$(\Delta u)'(r)>0, \qquad  \forall \ r>0.$$
Hence $\Delta u (r)$ is a strictly increasing function of $r$. By the relation between $u$ and $v$, we easily see that
$$\lim_{r \to \infty} \Delta u (r)=0.$$
So, we have
\begin{equation}
\label{3.13}
\Delta u (r)<0, \; \forall r>0.
\end{equation}
The strong maximum principle implies that
there can not be any $r_i>0$ such that $u (r_i)=0$. So $v (t)$
is a positive periodic solution to \eqref{2.2}. The proof of Theorem \ref{t1.2} is completed.
\qed

\section{Positive scalar curvature: Proof of Theorem 1.4}
\setcounter{equation}{0}

We present the proof of Theorem \ref{t1.3} in this section.
The proof of Theorem \ref{t1.3} is  based on iteration arguments, which is inspired by \cite{FWX}, where the authors consider the bounded solutions of $\Delta^2 u=|x|^a u^p$ in $\R^n$ with $a\geq 0$, $p>1$ and $n\ge 5$. As we have noticed before, the solutions in \cite{FWX} are bounded which will lead to good estimates in section 2 of \cite{FWX}, while in our case, the solutions admit singularity at $x=0$ which causes some difficulties in {\it a priori} estimates. Since our entire solutions of \eqref{1.1} are radially symmetric with precisely asymptotic behaviors at $r=0$ and $\infty$, we can still obtain the necessary estimates and the first step in the iteration arguments by using our entire solutions. Note that the first step in the iterations in \cite{FWX} is a direct result of \cite{PH}.

 For the readers' convenience,  let us explain the main idea in \cite{FWX}. Define a sequence of functions $(w_k)_{k=-1}$ with the form
$$
w_k:=\Delta u +\alpha_k|\nabla u|^2(u+\epsilon)^{-1}+\beta_k u^{\frac{p+1}{2}},
$$
where $\alpha_k$ and $\beta_k$ are certain nondecreasing sequences of nonnegative numbers with $\alpha_{-1}=0$ and $\beta_{-1}=0$. First it should be proved that $w_{-1}=\Delta u \le 0$. Then for the purpose of the iteration arguments, one should show that $w_0=\Delta u + \sqrt{\frac{2}{p+1}}|x|^{\frac{a}2}u^{\frac{p+1}2}\le 0$. Assuming that $w_k\leq 0$ holds, we construct a differential inequality for $w_{k+1}$ from which  we can prove  that $w_{k+1}\leq 0$ by  certain maximum principle type arguments. Choosing suitable sequences $\alpha_k, \beta_k$ and letting $k$ tend to infinity, we will have inequality (\ref{bs}).

Let's come back to our case. Define
\begin{equation}\label{wk}
w_k:=\Delta u +\alpha_k|\nabla u|^2(u+\epsilon)^{-1}+\beta_k u^{\frac{n}{n-4}},
\end{equation}
then by \eqref{3.13} $w_{-1}=\Delta u<0$. In the following we will show $w_0 \le 0$, the first step of the iteration arguments and some kind of maximum principle through the following lemmas. Once these are done, the remaining things are the same. For more details, one may refer to
\cite{FWX}.

\begin{lem}
\label{l4.1}
Let $v(t)=e^{\frac{n-4}{2}t}u(e^t)$ be given by Theorem \ref{t1.2}. Then the following inequality holds
$$
-\Delta u\geq \sqrt{\frac{n-4}{n}}u^{\frac{n}{n-4}} \qquad \mbox{in} \quad R^n \backslash \{0\}.
$$
\end{lem}

\begin{proof}
A simple calculation gives that
$$
\Delta u=e^{-\frac{n}{2}t}\left[v''(t)+2v'(t)-\frac{n(n-4)}{4}v\right],
$$
which yields
\begin{equation}\label{u1}
\Delta u+\sqrt{\frac{n-4}{n}}u^{\frac{n}{n-4}}=e^{-\frac{n}{2}t}\left[v''(t)+2v'(t)-\frac{n(n-4)}{4}v+\sqrt{\frac{n-4}{n}}v^{\frac{n}{n-4}}\right].
\end{equation}
Let $\tilde{v}(t)=v''(t)+2v'(t)-\frac{n(n-4)}{4}v+\sqrt{\frac{n-4}{n}}v^{\frac{n}{n-4}}$. We need to show $\tilde{v}\leq 0$ in $\R$.
Note that $\tilde{v}$ is smooth and periodic in $\R$ owning to $v$ is smooth, positive and periodic in $\R$.
On the maximum points $\{t_i\}$ of $v$, $\tilde{v}(t_i) \le 0$ since $v'(t_i)= 0, v''(t_i)\le 0, 0<v\leq L:=\Big(\frac{n}{n-4} \Big)^{\frac{n-4}{8}} K_0^{\frac{n-4}{8}}=\left(\frac{n^3(n-4)}{16}\right)^{\frac{n-4}{8}}$.
On the other hand, a simple calculation yields
\begin{eqnarray}
\label{4.13}
& &\tilde{v}''-2\tilde{v}'-\left[\frac{n(n-4)}{4}+\sqrt{\frac{n-4}{n}}\frac{n}{n-4}v^{\frac{4}{n-4}}\right]\tilde{v} \nonumber \\
& &\;\;\;\;\;\;\;\;\;\;=n\sqrt{\frac{n-4}{n}} v^{\frac{8-n}{n-4}}\left[v-\frac{2}{n-4}v'\right]^2\geq 0.
\end{eqnarray}
Now the maximum principle shows that $\tilde{v}\leq 0$ in $\R$.
\end{proof}

To apply the iteration argument, we also need to develop a maximum principle argument for the following equation
\begin{equation}
\label{4.14}
\Delta w-2\alpha(u+\epsilon)^{-1}\nabla u\cdot \nabla w+\alpha w \frac{|\nabla u|^2}{(u+\epsilon)^{2}}-\beta \frac{n}{n-4} u^{\frac{8}{n-4}} w=f(x)\geq 0\quad\text{in} \  \R^n\setminus{\{0}\}
\end{equation}
where $\alpha$, $\beta$ are positive constants, $u$ is a positive solution of equation \eqref{1.1} and $w$, $f\in C^{\infty}(\R^n\setminus{\{0}\}).$  For this, we first cite Lemma 4.1 in \cite{FWX}.

\begin{lem}
\label{l4.2}
Suppose that $w$ is a solution of the equation \eqref{4.14} where $u$ is a solution of \eqref{1.1} and
\begin{equation}\label{4.15}
w=\Delta u +\alpha (u+\epsilon)^{-1}|\nabla u|^2+\beta u^{\frac{n}{n-4}}
\end{equation}
for positive constants $\epsilon$, $\alpha$ and $\beta$. Then assuming that $\alpha<\frac{n}{n-4}$, the following holds
\begin{equation}\label{4.16}
\Delta \tilde{w}\geq 0 \qquad \text{on} \ \ {\{w\geq 0}\}\subset \R^n\setminus{\{0}\}
\end{equation}
where $\tilde{w}=(u+\epsilon)^t w$ for $t=-\alpha$.
\end{lem}

Now we apply Lemma \ref{l4.2} to show that $w$, the solution of \eqref{4.14}, is non-positive.

\begin{lem}
\label{l4.3}
Suppose that $\tilde{w}$ and $w$ are the same as Lemma 4.2 and $v(t)=e^{\frac{n-4}{2}t} u(e^t)$ is given by Theorem \ref{t1.2}, then $w\leq 0$.
\end{lem}

\begin{proof}

Since $v\in C^4(\R)$ is a periodic function with $0<v\leq L$ and $u=r^{-\frac{n-4}{2}} v$, we see that $0<u\leq L r^{-\frac{n-4}{2}}$, $w \leq Cr^{-\frac{n}2}$ and $\tilde {w}\leq C r^{-\frac{n}2+\frac{n-4}{2}\alpha}$ for some constant $C$ independent of $r$. Note that $\int_{\mathbb{B}_R} \Delta \tilde{w} \tilde{w}^s_+  <+\infty$ for $0<s<\frac{n-4}{n}$ and it follows from Lemma \ref{l4.2} that
$$
0\leq\int_{\mathbb{B}_R} \Delta \tilde{w} \tilde{w}^s_+ =-s\int_{\mathbb{B}_R} |\nabla{\tilde{w}_+}|^2\tilde{w}^{s-1}_+  +R^{n-1}\int_{\mathbb{S}^{n-1}} \tilde{w}_r\tilde{w}^s_+ .
$$
Therefore,
\begin{equation}\label{4.17}
\int_{\mathbb{B}_R} |\nabla{\tilde{w}_+}|^2\tilde{w}^{s-1}_+\leq \frac{1}{s(s+1)}R^{n-1}\int_{\mathbb{S}^{n-1}} (\tilde{w}^{s+1}_+)_r=C(s)R^{n-1}I'(R)
\end{equation}
where
$$
I(R):=\int_{\mathbb{S}^{n-1}} \tilde{w}^{s+1}_+=\int_{\mathbb{S}^{n-1}} (u+\epsilon)^{-(s+1)\alpha}{w}^{s+1}_+
$$
and $C(s)$ is a constant independent of $R$. Note that $w$ given as $w=\Delta u + \alpha|\nabla u|^2 (u+\epsilon)^{-1}+\beta u^{\frac{n}{n-4}}$ satisfies $w\geq 0$ if and only if $-\Delta u \leq \alpha|\nabla u|^2 (u+\epsilon)^{-1}+\beta u^{\frac{n}{n-4}}.$ Therefore,
$$
w^{s+1}_+\leq C|\nabla u|^{2(s+1)} (u+\epsilon)^{-(s+1)}+C u^{(s+1)\frac{n}{n-4}}
$$
where $C=C(\alpha,\beta,s)$. Hence, with $u(R)=R^{-\frac{n-4}{2}}v(\ln R)$,
$$
\begin{aligned}
I(R)&\leq C\int_{\mathbb{S}^{n-1}} u^{-(s+1)(\alpha+1)}|\nabla u|^{2(s+1)}+C\int_{\mathbb{S}^{n-1}}u^{-(s+1)\alpha}u^{(s+1)\frac{n}{n-4}}\\
&\leq C R^{\eta_1}+C R^{\eta_2}
\end{aligned}
$$
where $C$ is independent of $R$, $\eta_1=(s+1)[\frac{n-4}{2}(\alpha+1)-(n-2)]<0$ and $\eta_2=(s+1)[\frac{n-4}{2}\alpha-\frac{n}{2}]<0$ owing to $\alpha<\frac{n}{n-4}$ and $0<s<\frac{n-4}{n}$. So $I(R)\rightarrow 0$ as $R \to \infty$. Since $I(R)$ is a positive function and converges to zero, there is a sequence $\{R_i\}$ with $R_i \to \infty$ as $i \to \infty$ such that $I'(R_i)$ is non-positive. Therefore, \eqref{4.17} yields
$$
\int_{\mathbb{B}_{R_i}} |\nabla{\tilde{w}_+}|^2\tilde{w}^{s-1}_+\leq 0.
$$
Hence, $\tilde{w}_+$ has to be a nonnegative constant $c$. If $c \neq 0$, the continuity of $\tilde{w}$ implies that $\tilde{w} \equiv c$. From $\tilde {w}\leq C r^{-\frac{n}2+\frac{n-4}{2}\alpha}$ and $0<\alpha<\frac{n}{n-4}$,
this constant $c$ can not be strictly positive. This contradicts to $c \neq 0$. So $c=0$ and $\tilde{w}_+=0$ and therefore $w_+=0$.
\end{proof}

\bigskip

\noindent {\bf Acknowledgements.} The research of the first author is supported by NSFC
(11571093). The research of the second author is supported by the Chinese Postdoctoral Science Foundation (2017M610234) and the research of the third author is supported by NSFC (11671144). The fourth author is supported by NSERC of Canada. The fourth author thanks  Mar Gonzalez for useful discussions.

\end{document}